\documentclass[12pt]{article}

\usepackage[cp866]{inputenc}
\usepackage{amsfonts,amsthm,amssymb}
\usepackage{xcolor}

\usepackage[T2A]{fontenc}

\usepackage{epsfig}
\usepackage{graphicx}
\catcode`\@=11
\@addtoreset{equation}{section}\catcode`\@=12

\begin{document}

\newtheorem{lm}{Lemma}
\newtheorem{theorem}{Theorem}
\newtheorem{df}{Definition}
\newtheorem{prop}{Proposition}
\newtheorem{rem}{Remark}

\begin{center}

{\Large \bf Variety of strange
pseudohyperbolic
attractors in three-dimensional generalized H\'enon maps}\\~\\
{\bf A.S. Gonchenko, S.V. Gonchenko} \\
Lobachevsky State University of Nizhny Novgorod, Russia\\
e-mail: \textsf{agonchenko@mail.ru,  sergey.gonchenko@mail.ru} \\~\\
\end{center}

\begin{abstract}
In the present paper we focus on the problem of the existence of strange pseudohyperbolic attractors for three-dimensional diffeomorphisms. Such attractors are genuine strange attractors in that sense that each orbit in the attractor has a positive maximal Lyapunov exponents and this property is robust, i.e. it holds for all close systems. We restrict attention to the study of pseudohyperbolic attractors that contain only one fixed point.  Then we show that three-dimensional maps may have only 5 different types of such attractors, which we call the discrete Lorenz, figure-8, double-figure-8, super-figure-8, and super-Lorenz attractors. We find the first four types of attractors in three-dimensional generalized H\'enon maps of form $\bar x = y, \; \bar y = z, \; \bar z = Bx + Az + Cy + g(y,z)$,
where $A,B$ and $C$ are parameters ($B$ is the Jacobian) and $g(0,0) = g^\prime(0,0) =0$.
\end{abstract}

\begin{itemize}
\item[] Keywords: strange attractor, pseudohyperbolicity, discrete Lorenz attractor,\\ bifurcations, Lyapunov exponents
\item[]
\end{itemize}

\section{Introduction}

The goal of this paper is to contribute to the development of
the mathematical theory of dynamical chaos, especially to that its part which relates to
the theory of strange attractors for flows of dimension $\geq 4$ and maps (dif\-feo\-mor\-phisms) of dimension $\geq 3$.
%

A significant number of fundamental results has been obtained up to now in the theory of strange attractors of two-dimensional maps and three-dimensional flows.
It is sufficient to recall
such directions in the mathematical theory of dynamical chaos as the theory of spiral chaos \cite{Sh65,Sh70,Sh86,ACT81a,ACT85}, the theory of Lorenz attractors \cite{ABS77,ABS82,Sh80},  the theory of quasiattractors \cite{H76,AfrSh83a,AfrSh83,BenCar}.
Naturally, all these results can be used for the study of chaotic dynamics of multidimensional systems, even if to go by the way of a direct generalization of these theories. For example, many considerable results in this direction were obtained in \cite{Vitolo}.

However,
as it was recently revealed, multi\-di\-men\-sio\-nal systems may possess strange attractors of new types, the so-called {\em wild hyperbolic attractors}, \cite{TS98,TS08}.
The main feature of these attractors is that they allow homoclinic tangencies (unlike hyperbolic and Lorenz attractors) but do not contain stable periodic orbits and, moreover, such orbits  do not also appear at perturbations.\footnote{Systems with wild hyperbolic attractors belong to Newhouse regions, i.e.
open (in $C^2$-topology)  regions from the space of dynamical
systems where systems with homoclinic tangencies are dense.
However, these tangencies are such that their bifurcations do not lead to a creation of stable periodic orbits \cite{GST93c,GST96,GST08}.
The term ``wild'' goes back to the Newhouse paper \cite{N79}, where a notion of  ``wild hyperbolic set'' was introduced for a  uniformly hyperbolic basic set whose stable and unstable invariant manifolds have persistent tangencies.}
Accordingly, the
wild hyperbolic attractors must be considered as {\em genuine} strange attractors (until recently, only hyperbolic and Lorenz strange attractors were considered genuine).

The theory of wild hyperbolic attractors was laid in the paper  \cite{TS98} by Turaev and Shilnikov, in which  an example of {\em wild spiral attractor} of  four-dimensional flow was also constructed.
This attractor contains a saddle-focus equilibrium with eigenvalues $-\lambda_{ss}, -\lambda_s \pm i\omega, \lambda_u$, where $\lambda_{ss} >  \lambda_s > 0$, $\lambda_u > 0$,  $\lambda_u -\lambda_s - \lambda_{ss}<0$ and  $\lambda_u -\lambda_s >0$.  Note that if some system would have a homoclinic loop of this saddle-focus, then, in general, neither the system itself nor any close system has stable periodic orbits in a small fixed neighbourhood of the loop
\cite{OvsSh86,OvsSh91}. In the Turaev-Shilnikov example, there are no stable periodic orbits also globally,
due to the fact that the attracting set possesses a {\em pseudo-hyperbolic} structure.
The exact definition is given in \cite{TS98,TS08}.

In the case of discrete dynamical systems (maps),
the pseudohyperbolicity means that, in the absorbing domain,
the map $T$ admits
an invariant (with respect to action of the derivative $DT$) splitting
of the phase space into transversal strongly-contracting
and central-unstable subspaces $E^{ss}$ and $E^{cu}$, respectively.
The instability on $E^{cu}$ means that, first, if on $E^{cu}$ there exists a contraction, then
it is exponentially weaker
than any contraction in  $E^{ss}$, and, second, the derivative  $DT$ exponentially expands volumes in $E^{cu}$. The latter property implies that any orbit on the attractor has a positive maximal Lyapunov exponent which prevents, for example, the emergence of stable periodic orbits at small perturbations.

Then, by definition \cite{TS98,TS08}, a {\em pseudohyperbolic attractor} is an asymptotically stable chain-transitive compact invariant set which carries a uniformly
pseudo-hyperbolic structure.\footnote{Recall that the chain-transitivity means that for any $\varepsilon > 0$ and any two points of the attractor there is an $\varepsilon$-orbit which contains these points. In turn, for a map $T: x_{n+1} = \psi(x_n)$, a set $\{x_i\}_{i\in Z}$ is called an orbit if $x_{i+1} = \psi(x_i)$ and $\varepsilon$-orbit if $|x_{i+1} - \psi(x_i)|<\varepsilon$ for each $i$.}

Note that if a three-dimensional map has a pseudohyperbolic attractor with $\dim E^{ss} =1$, then the Lyapunov exponents $\Lambda_1>\Lambda_2>\Lambda_3$ of the attractor must satisfy the condition
\begin{equation}
\Lambda_1>0,\; \Lambda_1 +\Lambda_2> 0,\;  \Lambda_1 +\Lambda_2 + \Lambda_3 < 0.
\label{Lyapcr}
\end{equation}
In this paper, condition (\ref{Lyapcr}) is checked numerically for each attractor under consideration. Other necessary conditions, which we also check, are more simple and can be  checked analytically. For example, for different reasons, we exclude from the class of pseudohyperbolic attractors those ones that either contain a saddle fixed point with a two-dimensional unstable invariant manifold or if this point is a saddle-focus, see below. On the other hand, if an attractor contains only one saddle fixed point $O$ such that $\dim W^u(O)=1$ and the stable multipliers are real, then condition (\ref{Lyapcr}), where $\Lambda_1,\Lambda_2,\Lambda_3$ have been found numerically,
is quite close to be sufficient for the pseudohyperbolicity of the attractor. Additionally, we would need to further verify that the subspaces $E^{ss}$ and $E^{cu}$ depend continuously on the point of the attractor
and, moreover, that angles between the subspaces do not tend to zero.
 We plan to give the corresponding computer assisted proofs in one of the nearest papers, here
we consider condition (\ref{Lyapcr}) as sufficient.

We notice that the pseudohyperbolicity
persists at small smooth perturbations. Moreover, as it is shown in \cite{TS08}, the properties of pseudo-hyperbolic flows are inherited by the corresponding Poincar\'e maps at small periodic perturbations.  Then a periodically perturbed system with the Lorenz attractor gives one more example of wild pseudohyperbolic attractor, see also paper \cite{Sat05} where it was proved that stable periodic orbits do not appear under such perturbations.
We will call such attractor, that exists in the Poincar\'e map,
as a {\em discrete super Lorenz attractor}.

This attractor has the following characteristic properties (inherited from the flow): it has a fixed point $O$  with the
multipliers $\lambda_1,\lambda_2,\lambda_3$ such that $\lambda_1>1, 0<\lambda_3 < \lambda_2<1$ and $\lambda_1\lambda_2>1$; the absorbing domain
is a ball with two holes containing two saddle-focus fixed points (with $\dim W^u =2$); the corresponding Poincar\'e map possesses the Lorenzian symmetry (like $x\to x, y\to -y, z\to -z$).

One more new type of pseudohyperbolic Lorenz-like attractors, which we call simply as a {\em discrete Lorenz attractor}, was discovered in \cite{GOST05}, where it was shown that the three-dimensional H\'enon map
\begin{equation}
\bar x = y, \; \bar y = z, \; \bar z = M_1 + M_2 y + Bx - z^2,
\label{3Hen1}
\end{equation}
where $(x,y,z)\in \mathbb{R}^3$ and $(M_1,M_2,B)$ are parameters ($B$ is the Jacobian of the map),
can possesses strange attractors that look to be very similar to the
discrete super Lorenz attractors,  see e.g. Fig.~\ref{AL-z2}. However, the corresponding characteristic properties will be quite different: the fixed point $O$ of attractor has multipliers   $\lambda_1,\lambda_2,\lambda_3$ such that $\lambda_1<-1,\; -1<\lambda_3 <0 < \lambda_2<1$, $\;|\lambda_3| < |\lambda_2|$ and $|\lambda_1\lambda_2|>1$; the absorbing domain
is a ball with two holes containing a saddle-focus period-2 cycle; the map has no symmetries, but the Lorenzian symmetry appears locally, for invariant manifolds of $O$,
due to the point $O$ has negative multipliers.

Thus, the discrete Lorenz attractor (DLA) and discrete super Lorenz attractor (SLA) are quite different. Nevertheless,
the connection between DLA and SLA is direct and evident.
Thus, it was established in \cite{SST93,GOST05} that DLA can be born as result of local bifurcations of a fixed point with the triplet  $(-1,-1,+1)$ of multipliers. In this case the {\em second iteration} of the map can be locally embedded (up to small non-autonomous periodic terms) into a 3d flow which coincides, in some rescaled coordinates, with the Shimizu-Morioka system \cite{ShimMor}. In turn, the latter system has the Lorenz attractor for certain values of parameters \cite{ASh93}.  Thus, by \cite{TS08},
at least for values of the parameters from some open domain adjoining to the point $(M_1 = -1/4, M_2=1, B=1)$, the discrete Lorenz attractor of map (\ref{3Hen1}) is genuine pseudo-hyperbolic one.
Then  we can say,
in principle, that both the similarity and difference between the discrete Lorenz attractor and the discrete super Lorenz attractor can be expressed by a formula ``DLA is a square root from SLA''.

Note that it was shown in \cite{GGS12}, see also \cite{GGKT14}, that discrete Lorenz attractors can arise as a result of simple and universal bifurcation scenarios which naturally occur  in one-parameter families of three-dimensional maps.
A sketch of such  scenario for one-parameter families $T_\mu$ of three-dimensional dif\-feo\-mor\-phisms is shown in Fig.~\ref{scen2}. 
%
Here, the way (a)$\Rightarrow$(b)$\Rightarrow$(c) corresponds to the appearance of discrete Lorenz attractor: first, the stable fixed point undergoes
a supercritical period doubling bifurcation and becomes a saddle with negative unstable multiplier $\lambda_1$ ($\lambda_1 < -1$); next, homoclinic orbits are created such that their homoclinic points on $W^s_{loc}(O_\mu)$ are located on one side of $W^{ss}_{loc}(O_\mu)$ -- this is a general situation when the stable multipliers $\lambda_2$ and $\lambda_3$, where $-1<\lambda_3 < 0 < \lambda_2 <1$, are such that $|\lambda_3| < |\lambda_2|$, see Fig.~\ref{scen2}(c).

However, one can also imagine, see the way (a)$\Rightarrow$(b)$\Rightarrow$(d),  that the  homoclinic points on $W^{s}_{loc}(O_\mu)$ are on opposite sides from $W^{ss}_{loc}(O_\mu)$ -- this is a general situation when  $|\lambda_3| > |\lambda_2|$, see Fig.~\ref{scen2}(d). Then the corresponding attractor, called {\em discrete figure-8 attractor},
will have shape different from the discrete Lorenz attractor (see e.g. Figs.~\ref{8attrnew2} and \ref{8quas}).

Note that the condition
\begin{itemize}
\item
the saddle value $\sigma \equiv |\lambda_1|\times \max\{|\lambda_2|,|\lambda_3|\}$ is greater than 1
\end{itemize}
is necessary for the discrete Lorenz or figure-8 attractors to have a pseudo\-hy\-per\-bo\-lic structure. Indeed, in the case  $\sigma <1$ the point $O_\mu$ would be not pseudohyperbolic and,
moreover, bifurcations of homoclinic tangencies (which invetably arise here) would lead
to the emergence of periodic sinks \cite{G83,GST93c,GOT12}. Thus, if $\sigma<1$,   both the discrete Lorenz and figure-8 attractors are, in fact, {\em quasiattractors} \cite{AfrSh83a}, i.e. such
attractors  which look chaotic but stable periodic orbits can emerge near attractor
at arbitrary small perturbations.
Since these periodic sinks have, in general, very big periods and small attracting domains, the numerically or experimentally observed
global dynamics can look perfectly chaotic --
therefore, quasiattractors can be considered as ``strange attractors from physical point of view''.

For the theory of dynamical chaos, it is important that discrete Lorenz and figure-8 attractors can be genuine
strange attractors, i.e. attractors for which the property \texttt{``any orbit of the attractor has a positive maximal Lyapunov exponent''} is open (in $C^1$-to\-po\-lo\-gy).
Moreover, such attractors (Lorenz and figure-8 ones) should
be encountered
\begin{figure}[ht]
\centerline{\epsfig{file=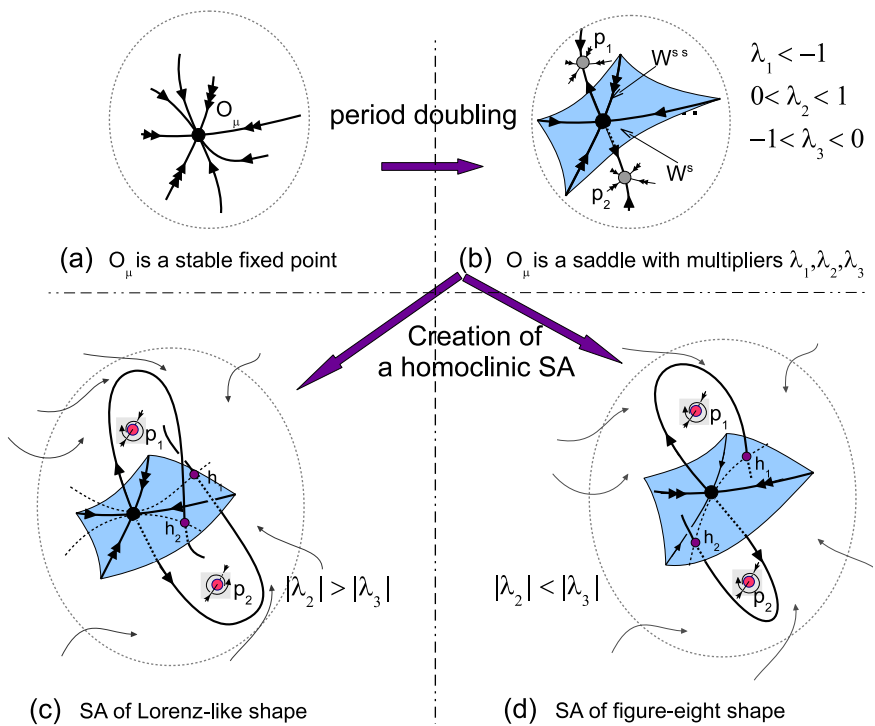, width=12cm
}}
\vspace{-1cm}
\caption{{\footnotesize A sketch of a bifurcation scenario leading to the appearance of either a discrete Lorenz attractor (the way (a)$\Rightarrow$(b)$\Rightarrow$(c)) or  a discrete figure-8 attractor (the way (a)$\Rightarrow$(b)$\Rightarrow$(d)).
The points $h_1$ and $h_2$ are two points of the same homoclinic orbit, i.e. $h_i\in W^s(O_\mu)\cap W^u(O_\mu)$ and $h_2 = T_\mu(h_1)$.  These points are located on one side of $W^{ss}_{loc}(O_\mu)$ in the Lorenz case (c) and on the opposite sides from of $W^{ss}_{loc}(O_\mu)$ in the figure-8 case (d).  }}
\label{scen2}
\end{figure}
widely in applications.
We already know several
such examples (though still not very much).
For instance,
discrete Lorenz attractors were found in three-dimensional H\'enon-like maps of various types \cite{GOST05,GGOT13,GGKT14} and in nonholonomic models of rattleback (called also a Celtic stone) \cite{GGK13,GG15}.
A discrete figure-8 attractor was first observed in a nonholonomic model of Chaplygin top~\cite{Kaz14}. These homoclinic attractors are {\em global} in the sense that there exist quite large absorbing domains in the phase space which  contain these attractors.\footnote{The birth of small discrete Lorenz attractors under global bifurcations of three-dimensional diffeomorphisms with nontransversal homoclinic and heteroclinic orbits was studied in \cite{GMO06,GST09,GO13,GOTat14}.}
Moreover, in the case of the pointed out three-dimensional quadratic H\'enon maps, the discrete Lorenz attractor seems to be only attractor of the map.

Concerning the discrete figure-eight attractors we note that they {\em have no flow analogs}, contrary to the discrete Lorenz-like attractors. Indeed, in order to have a figure-8 attractor in the Poincar\'e map of periodically forced flow, we need to take a 3d flow that has a homoclinic figure-eight of a saddle equilibrium with the eigenvalues $\nu_3 < \nu_2 <0 < \nu_1$, where
$\nu_1 + \nu_2>0$. However, in this case
the homoclinic figure-eight is not stable in general and, hence, it can not become a strange attractor of the corresponding Poincar\'e map.

The discrete Lorenz attractors as well as the discrete figure-eight attractors are novel genuine strange attractors
for diffeomorphisms of dimension $\geq 3$.
Therefore, in the present paper we focus on the study of such attractors. We consider a partial problem on the existence of discrete Lorenz and figure-eight attractors, as well as pseudohyperbolic attractors of other types, in three-dimensional generalized H\'enon maps of the form
\begin{equation}
\bar x = y, \; \bar y = z, \; \bar z = Bx + f(y,z)
\label{3HenG}
\end{equation}
with $0<B<1$.
%
By definition, the Lorenz and figure-eight attractors contain a saddle fixed point and we assume that this fixed point is in the origin, i.e. it is the point $O(0,0,0)$.
Thus, we can rewrite
map (\ref{3HenG}) in the form
\begin{equation}
\bar x = y, \; \bar y = z, \; \bar z = Bx + Az + Cy + \tilde f(y,z),
\label{GHM1-8}
\end{equation}
where $A$ and $C$ are some parameters and $\tilde f(0,0) = \tilde f^\prime(0,0) =0$.

For the numeric search of strange attractors we apply bifurcation methods
together with the construction of {\em colored Lyapunov diagrams} and {\em saddle charts}.

The method of colored Lyapunov diagrams is well-known and it consists in de\-ter\-mi\-na\-tion of parameter domains with different sets of Lyapunov exponents $\Lambda_1,\Lambda_2,\Lambda_3$ of orbits that do not leave (for forward iterations) some region of the phase space.
The cases with $\Lambda_1>0$ corresponds to strange attractors. In this paper we paid the main attention to the cases when these attractors are {\em homoclinic}, i.e. when they contain the saddle fixed point $O(0,0,0)$ of map (\ref{GHM1-8}). For this goal we improve the Lyapunov diagram adding to it a domain of parameter values (painted in {\em dark grey})
where the attractor contains the point $O$. Numerically, the corresponding domain is defined
when, first, $\Lambda_1> 0$ and, second, when the distance between some points of the attractor and the point $O$ is smaller than some $\varepsilon$ (e.g. for $\varepsilon = 10^{-4}$).

The method of saddle charts
consists in construction of a decomposition of the parameter space (e.g., we consider, for map (\ref{GHM1-8}), the $(A,C)$-parameter planes for fixed values of  $B$) into domains corresponding to different types of the point $O$. We also distinguish the cases when
the saddle value of $O$ is less and greater than one as well as the cases when positive and negative stable multipliers are leading (nonleading).\footnote{Note that analogous saddle charts with respect to equilibria of three-dimensional flow were constructed in \cite{book2}, see Part II, Appendix C there.}
In Section~\ref{HomAttHen} we construct the saddle charts for the fixed point $O(0,0,0)$ of map (\ref{GHM1-8});
an example of such saddle chart on the $(A,C)$-plane is shown in Fig.~\ref{fig:8cond1} for $B=0.5$.
Evidently, the saddle chart of point $O$
is completely independent of the nonlinear terms $\tilde f(y,z)$ of map (\ref{GHM1-8}). It is not the case for the Lyapunov diagram (constructed on the same $(A,B)$-coordinate plane) which depends essentially on these terms. Thus, having together the saddle chart and the Lyapunov diagram (on the same $(A,C)$-parameter plane) we can visually define which strange attractors can be homoclinic and, moreover, we can predict their type.

One of goals
of this paper is to provide new examples of pseudohyperbolic attractors of three-dimensional maps. We restrict this problem to such  
attractors of map (\ref{GHM1-8}) which contain the fixed point $O(0,0,0)$. In this case we can also essentially narrow the field of study using the following considerations. First, if the attractor is pseudohyperbolic, then the fixed point is also pseudohyperbolic. In the case $\dim W^u(O) =1$, this implies that $\sigma >1$. Second, since the attractor is wild hyperbolic,  inevitably arising homoclinic tangencies must be such that their bifurcations do not lead to creation of periodic sinks. Then we automatically exclude cases where the fixed point is a saddle-focus \cite{GST93c,GST08}. Third, in the cases where the point $O$ is a saddle with $\dim W^u(O)=2$, the corresponding homoclinic attractor should be also quasiattractor. However, this fact is more heuristic than absolutely true. The point is that the two-dimensional unstable manifold forming an attractor must be folded in several directions. Then these ``multi-folds'' must produce, in the three-dimensional case, the so-called non-simple homoclinic tangencies \cite{Tat01,GGT07} which, in turn, lead to the appearance of stable periodic orbits.

Thus, we must consider only cases where $\dim W^u(O) =1$ and  $\sigma >1$. Such cases are only four depending on a set of multipliers $\lambda_1,\lambda_2,\lambda_3$ of $O$:
\begin{enumerate}
\item[P1.] $\lambda_1 <-1, 0<\lambda_2< 1, -1<\lambda_3<0$,$ \lambda_2> |\lambda_3|$, and $\sigma>1$;
\item[P2.] $\lambda_1 <-1, 0<\lambda_2< 1, -1<\lambda_3<0$, $|\lambda_3|> \lambda_2$, and $\sigma>1$;
\item[P3.] $\lambda_1 > 1, -1<\lambda_2 <\lambda_3<0$ and $\sigma>1$;
\item[P4.] $\lambda_1 >1, 0<\lambda_3 <  \lambda_2<1$ and $\sigma>1$;
\end{enumerate}

In Section~\ref{Ex8Lor} we give some examples of the corresponding discrete pseudohyperbolic attractors.
%
For example, for map (\ref{GHM1-8}) with $\tilde f = -z^2$ and $B=0.7$ (in fact, for map (\ref{3Hen1}))   the picture of
the Lyapunov diagram super\-impo\-sed on the saddle chart of the point $O(0,0,0)$ is shown in Fig.~\ref{AL-z2}(a).
We see here that the dark grey part of the diagram intersects with the domain \textsf{LA} of the chart corresponding to the case P1 of multipliers of point $O$.
Thus, we have situation when the corresponding homoclinic attractors should be the discrete Lorenz ones.
And indeed this is the case: the projections (into the $(x,y)$-plane) of phase portraits of two such attractors
are shown in Figs.~\ref{AL-z2}(b) and~\ref{AL-z2}(d); condition (\ref{Lyapcr}) holds for numerically obtained Lyapunov exponents (in both cases  $\Lambda_1 + \Lambda_2 \approx 0.02$); and
the stable and unstable invariant manifolds of $O$ have homoclinic intersections, see Figs.~\ref{AL-z2}(c),~\ref{AL-z2}(e) and~\ref{AL-z2}(f). Thus, both attractors from Fig.~\ref{AL-z2} are, in fact, wild pseudohyperbolic attractors. Note that the second attractor, from Fig.~\ref{AL-z2}(d), has a lacuna, i.e. a specific hole inside the attractor containing saddle invariant curves without homoclinics. Note that the flow Lorenz attractors with lacunas (containing saddle limit cycles) exist in the Lorenz model \cite{ABS82}.

We give also three more examples of discrete pseudohyperbolic attractors of new types: (i) a discrete figure-8 attractor, see Fig.~\ref{8attrnew2},  for map (\ref{GHM1-8}) with $\tilde f = -y^2 + 0.515 yz -1.45 z^2$ and $B=0.72$ -- case P2 of multipliers of $O$; (ii) a discrete double figure-8 attractor, see Fig.~\ref{d8attr}, for map (\ref{GHM1-8}) with cubic nonlinearity $\tilde f = -2y^3 - 2.25 z^3$ and $B=0.5$ -- case P3 of multipliers of $O$; (iii) a discrete super figure-8 attractor, see  Fig.~\ref{sup8b}, for map (\ref{GHM1-8}) with cubic nonlinearity $\tilde f = y^3 - z^3$ and $B=0.05$ -- case P4 of multipliers of $O$.
In all these cases, condition (\ref{Lyapcr}) holds.
We include ``figure-8'' in the name of these attractors because of the configuration of their unstable manifold $W^u(O)$, compare Figs.~\ref{scen2}(d) and  \ref{spirat}.

In Section~\ref{Exstrattr} we consider several interesting examples of homoclinic quasiattractors in maps of form (\ref{GHM1-8}).
The first example is given by a ``discrete figure-8-like'' attractor of map (\ref{GHM1-8}) with $\tilde f = -y^2 - 1.45 z^2$ and $B=0.7$, see Fig.~\ref{8quas}. The attractor contains the saddle $O$ with multipliers corresponding to the case P2. Despite this fact, the necessary condition (\ref{Lyapcr}) is not fulfilled:  we calculate that $\Lambda_1 + \Lambda_2 \approx -0.03 <0$ in this case. Therefore, we conclude that this attractor is a quasiattractor, since the inequality $\Lambda_1 + \Lambda_2 <0$ implies that the attractor contains saddle periodic orbits with saddle values less than 1 and, thus, bifurcations of the corresponding homoclinic tangencies have to lead to the birth of periodic sinks. This example shows that the property of pseudo-hyperbolicity of homoclinic attractor is more delicate than simply property of pseudo-hyperbolicity of its fixed points. It is also true in the case of homoclinic attractors with $\dim W^u(O)=2$. Two examples of such attractors of map (\ref{GHM1-8}) with $\tilde f = -y^3$  and $B=0.5$ are shown in Fig.~\ref{book-attr}.

In Section~\ref{spirtype}, we give two examples of spiral quasiattractors, i.e. attractors containing point $O$ that is a saddle-focus. In Fig.~\ref{spir1}, a discrete spiral attractor is shown for map (\ref{GHM1-8}) with $\tilde f = 0.5y^3 - 6 zy^2 + 0.5 z^3$ at $B=0.5$. The attractor contains the saddle-focus $O$ of type (2,1), i.e. with multipliers such that $\lambda_1>1,\lambda_{2,3}=\rho e^{\pm i\varphi}$ and $0<\rho<1$. In Fig.~\ref{shilattr}, a discrete Shilnikov attractor is shown for map (\ref{GHM1-8}) with $\tilde f = -y^2$ and $B=0.5$. The attractor contains the saddle-focus $O$ of type (1,2), i.e. with multipliers such that $\lambda_{1,2} =\gamma e^{\pm i\psi}$, where $\gamma>1$, and $0<\lambda_3 <1$. We give also in Section~\ref{spirtype} some small background related to the Shilnikov spiral attractors. In particular, we briefly discuss a scenario of the emergence of a discrete Shilnikov attractor (this scenario is illustrated in Fig.~\ref{shildiscr}).

In Section~\ref{HenType} we consider the limit case $B=0$ of map (\ref{GHM1-8}). In this case, the saddle chart of $O$ takes a specific form, see Fig.~\ref{fig:8cond1(0p0)}. We examine here the partial case when map (\ref{GHM1-8}) becomes the H\'enon map at $B=0$ (i.e., in fact, the case of map (\ref{3Hen1}) with $B=0$). Then we find that the H\'enon map can have strange attractors
(with a fixed point)
only of two types -- by our terminology, they are either a ``discrete Lorenz quasiattractor'' (when the Jacobian of map is negative) or a ``discrete figure-8 quasiattractor'' (when the Jacobian is positive). Some examples of the corresponding attractors are shown in Fig.~\ref{fig:HAandLQ}. In particular, the discrete Lorenz quasiattractor from Fig.~\ref{fig:HAandLQ}(a) is the famous H\'enon attractor \cite{H76}.

\section{On the saddle charts for
three-dimensional generalized H\'enon maps.} \label{HomAttHen}

We consider map (\ref{GHM1-8}).
The characteristic polynomial at the point $O(0,0,0)$ has the form
\begin{equation}
\chi(\lambda)\equiv\lambda^3 - A \lambda^2 - C \lambda - B = 0.
\label{GHM3-8}
\end{equation}
Then we can consider coefficients $A,B,C$ as parameters that control bifurcations of point $O$.
Indeed, in the $(A,B,C)$-parameter space, there are the following
codimension one bifurcation surfaces:
\begin{equation}
L^+ :\; A+B+C = 1
\label{eq:lm1-1(8)}
\end{equation}
when point $O$ has the multiplier $+1$,
\begin{equation}
L^- : \; A+B-C = - 1
\label{eq:lm1-2(8)}
\end{equation}
when point $O$ has the multiplier $-1$, and
\begin{equation}
L_\varphi :\; + B(A-B) +C = 0, \;\; -2 < A-B <2
\label{eq:lm1-3(8)}
\end{equation}
when point $O$ has the  multipliers $\lambda_{1,2} = e^{\pm i\varphi}$, where $0<\varphi<\pi$.

\subsection{Conditions for existence of discrete Lorenz attractors.}

As was said in the Introduction, if map (\ref{GHM1-8}) possesses a discrete Lorenz attractor containing the fixed point $O(0,0,0)$, then
the multipliers $\lambda_{1},\lambda_{2},\lambda_{3}$ of this point satisfy the following inequalities
\begin{equation}
\begin{array}{l}
a) \qquad \lambda_{1} < -1,\; -1 <\lambda_{3}< 0,\; \lambda_2 > 0, \\
b)\qquad 0 < \lambda_2 < 1,  \\
c) \qquad \lambda_{2} > |\lambda_{3}|, \\
d) \qquad \lambda_{1}\lambda_{2} < -1 .
\end{array}
\label{GHM4-8}
\end{equation}

\begin{lm}
Let fixed point $O(0,0,0)$ of map {\rm (\ref{GHM1-8})} have the multipliers satisfying
{\rm (\ref{GHM4-8})} for $B>0$. Then one has
\begin{equation}
\begin{array}{l}
(a) \qquad C > A + B + 1, \\
(b)\qquad C < 1-B - A, \\
\displaystyle (c) \qquad C  > -\frac{B}{A} ,\; A<0,  \\ 
%
(d) \qquad C > 1 + AB + B^2 .
\end{array}
\label{lm2-8(1)}
\end{equation}
\label{lm2(8)}
\end{lm}

\begin{proof}
Noth that since
$B>0$, the fixed point $O$ has always a positive multiplier $\lambda_2$ and also $\chi(0)=-B<0$. Then, if inequalities (\ref{GHM4-8}a) hold, we have that $\chi(-1) >0$ and, thus,
there exists only by one root (no more than one, since the polynomial is cubic)
in each interval $(0;+\infty)$, $(-1;0)$ and $(-\infty;-1)$. Since $\chi(-1) = -1 - A + C - B$, the condition $\chi(-1) >0$ implies (\ref{lm2-8(1)}a).

\begin{figure}[ht]
\centerline{\epsfig{file=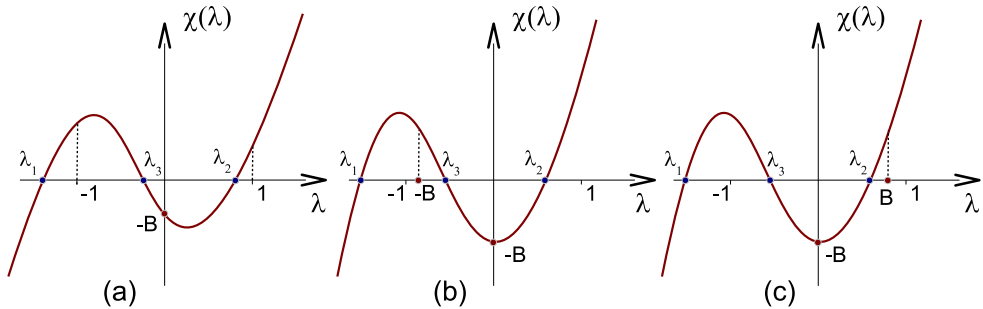, width=14cm
}} \caption{
{\footnotesize Geometric interpretation of the fact that for $B>0$,
the characteristic polynomial (\ref{GHM3-8}) has the roots  $\lambda_{1},\lambda_{2},\lambda_{3}$ such that
(a) $\lambda_{1}<-1, -1<\lambda_{3}<0, 0<\lambda_{2}<1$ ;  (b) $ -B<\lambda_{3}<0$, i.e. $\lambda_1\lambda_2 < -1$ и
(c) $ 0<\lambda_{2}<B$, i.e. $\lambda_1\lambda_3 > 1$. }}
\label{fig:harur1}
\end{figure}

It follows also from (\ref{GHM4-8}a) and (\ref{GHM4-8}b) for $B>0$ that   $\chi(+1) >0$, see Figure~\ref{fig:harur1}(a), i.e.  inequality
(\ref{lm2-8(1)}b) holds.

Let inequality (\ref{GHM4-8}c) hold. The limiting case $\lambda_2 = - \lambda_3 = \lambda_0$ means that $\lambda_1 = A$ and, hence,
 $\chi(\lambda)=(\lambda - A)(\lambda^2 - \lambda_0^2)$. By virtue of (\ref{GHM3-8}), we get that $\lambda_0^2 = C$ and, thus, $AC + B = 0$.
We show that if
\begin{equation}
AC + B < 0 ,
\label{GHM4-8***}
\end{equation}
then $\lambda_2 > |\lambda_3|$. It is obvious that the inequalities  (\ref{GHM4-8***}) and (\ref{lm2-8(1)}c) are equivalent for $A<0$.

Let $AC + B = -\varepsilon$, i.e. $B = -AC - \varepsilon$, where $\varepsilon>0$ is sufficiently small. Then the characteristic polynomial
(\ref{GHM3-8}) can be written in the form
\begin{equation}
\chi(\lambda) = (\lambda-A)(\lambda^2 -C)  + \varepsilon = 0.
\label{GHM5-8}
\end{equation}
Note that we have $A<-1$ and $0<C<1$ for $\varepsilon=0$, since we assume that the conditions (\ref{GHM4-8}a) and (\ref{GHM4-8}b) for
the multipliers of the point  $M$ are satisfied. Then equation (\ref{GHM5-8}) has three roots close to
$\lambda = \left(A, \sqrt{C}, - \sqrt{C}\right)$.

Let $\lambda_2 = \sqrt{C} + \delta_2$ be a solution of equation (\ref{GHM5-8}). Then for $\varepsilon$ small enough, we get from (\ref{GHM5-8})
that
$$
\delta_2 = \frac{- \varepsilon}{2\sqrt{C}(\sqrt{C} + |A|)} + O(\varepsilon^2).
$$
Similarly, if $\lambda_3 = - \sqrt{C} + \delta_3$ is a solution of  (\ref{GHM5-8}), then
$$
\delta_3 =  \frac{\varepsilon}{2\sqrt{C}(|A| -\sqrt{C})} + O(\varepsilon^2).
$$
Since $A<-1, 0<C<1$, we obtain that $|\delta_2| < |\delta_3|$ and, hence, $\lambda_2 > |\lambda_3|$.

Finally, since $\lambda_1\lambda_2\lambda_3 = B$, it implies from (\ref{GHM4-8}d) that $\lambda_3>-B$. The last inequality implies that $\chi(-B)>0$,
see Figure~\ref{fig:harur1}(b), i.e. $-B^3 - AB^2 + BC - B>0$.  For $B>0$, this inequality is equivalent to (\ref{lm2-8(1)}d). 
\end{proof}

\subsection{Conditions for existence of discrete figure-8 attractors.}

As mentioned in the Introduction, if  a figure-8 attractor exists in map (\ref{GHM1-8}) for $0<B<1$, then the fixed point of attractor should have
multipliers satisfying
\begin{equation}
\begin{array}{l}
a) \qquad \lambda_{1} < -1,\; -1 <\lambda_{3}< 0,\; \lambda_2 > 0, \\
b)\qquad 0 < \lambda_2 < 1,  \\
c) \qquad \lambda_{2} < |\lambda_{3}|, \\
d) \qquad \lambda_{1}\lambda_{3} >1 .
\end{array}
\label{GHM6-8}
\end{equation}

\begin{lm}
Let fixed point $O(0,0,0)$ of map {\rm (\ref{GHM1-8})} have the multipliers satisfying
{\rm (\ref{GHM6-8})} for $B>0$. Then one has
\begin{equation}
\begin{array}{l}
(a) \qquad C > A + B + 1, \\
(b)\qquad C < 1-B - A, \\
\displaystyle (c) \qquad C  < -\frac{B}{A} ,\; A<0,  \\ 
%
(d) \qquad C <  -1 - B A + B^2 .
\end{array}
\label{lm3-8(1)}
\end{equation}
\label{lm3(8)}
\end{lm}

\textit{Proof}.
Due to Lemma~\ref{lm2(8)}, it suffices to derive only condition (\ref{lm3-8(1)}d). Indeed, the inequalities (a) and (b) of (\ref{GHM6-8}) and (\ref{GHM4-8}) are
the same,
whereas the inequalities (c)
are opposite.
Since $\lambda_{1}\lambda_2\lambda_{3} = B$, it follows from condition $\lambda_{1}\lambda_{3} >1$ that $\lambda_2 <B$. It is evident
that the last inequality implies $\chi(B)>0$, see Figure~\ref{fig:harur1}(c), i.e. $B^3 - AB^2 - BC - B>0$, which gives inequality
(\ref{lm3-8(1)}d) for $B>0$. $\;\;\Box$

\section{Construction of the complete saddle chart of $O(0,0,0)$.
}  \label{TypeAttr}

We notice that, for any fixed $B$,
inequalities (\ref{lm2-8(1)}) and (\ref{lm3-8(1)})
define, in the corresponding $(A,B)$-parameter plane, the domains  \textsf{LA} and \textsf{8A},
where the saddle fixed point $O(0,0,0)$
is of the suitable type in order to map (\ref{GHM1-8}) could have
either a discrete Lorenz attractor or a discrete figure-8 attractor,
respectively. Besides, knowing the geometric sense of the inequalities (\ref{lm2-8(1)}) and (\ref{lm3-8(1)}) and of their opposites, we can
construct (for a given fixed $B$) a partition of the $(A,C)$-plane into the domains corresponding to different types of the fixed point
$O(0,0,0)$. We call such a partition as {\em a saddle chart}.

\begin{figure}[ht]
\centerline{\epsfig{file=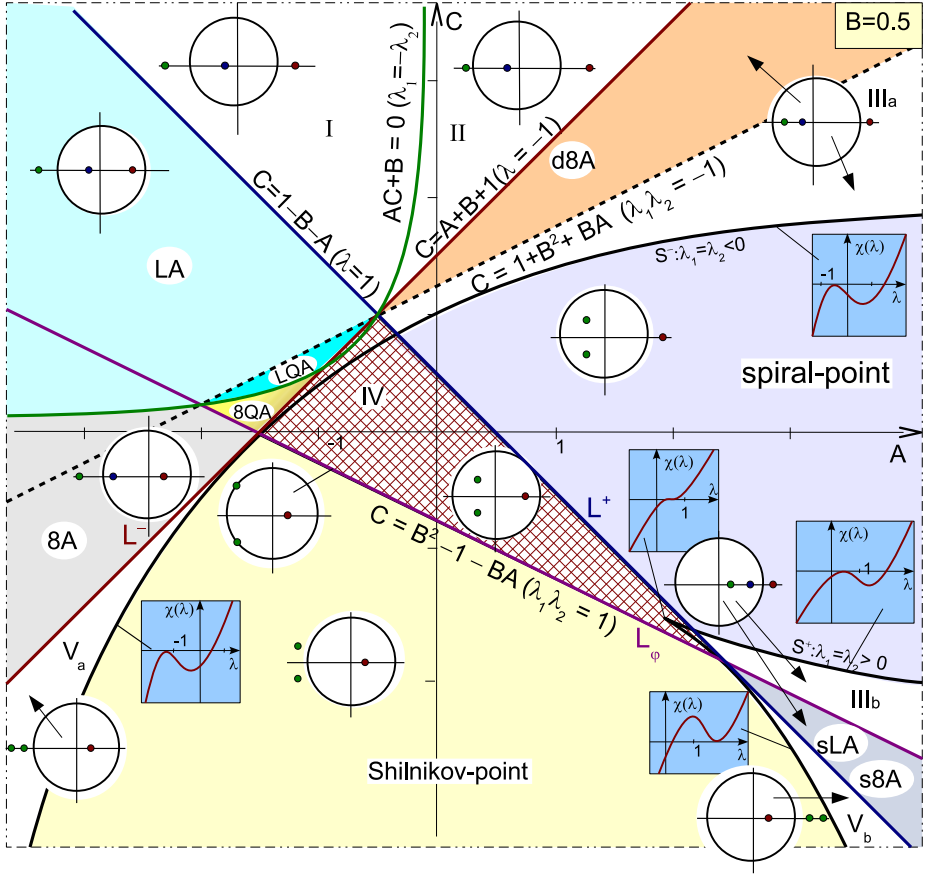, width=14cm
}}
\vspace{-1cm}
\caption{
{\footnotesize The ``saddle chart'' for the map (\ref{GHM1-8}) at $B=0.5$.
}
}
\label{fig:8cond1}
\end{figure}

In Figure~\ref{fig:8cond1}, such a saddle chart is shown for the case $B=0.5$.
In this chart we select the domain \textsf{IV}, the so-called ``stability triangle''
(the domain $\{C> B^2 -1 -BA\}\cap\{C< A+B+1\}\cap\{C < 1-B-A\}$), related to the case where the fixed point  $O(0,0,0)$ is asymptotically stable.
For all other values of parameters $A$ and $C$ (except for those on the bifurcation curves $L^+,L^-$ and $L_\varphi$) the point $O(0,0,0)$ is saddle and, thus, it has
multipliers inside and outside the unit circle (the location of multipliers is also shown in Figure~\ref{fig:8cond1}).
The boundaries of the domains in the saddle chart are seven main curves. They contain, first, the  bifurcation curves $L^+$, $L^-$ and $L_\varphi$ given by equations (\ref{eq:lm1-1(8)})--(\ref{eq:lm1-3(8)}).
Notice that the curve $L_\varphi$ is only a part of the boundary curve $C = B^2 - 1 - BA$ which
is not a bifurcation curve for
$|A-B| \geq 2$, and, moreover,
the multipliers of $O$ are $(B, -|\lambda|, -|\lambda|^{-1})$) for $A-B \leq -2$ and $(B, |\lambda|, |\lambda|^{-1})$) for $A-B \geq 2$.

In the saddle chart there are also the following four boundary curves:
\begin{itemize}
\item
the curve $C = 1 + B^2 + B + A$, when $\lambda_1\lambda_2 = -1$,
\item
the curve $AC + B = 0, \;\; A <0\;$, when the resonance $\lambda_1  = -\lambda_2$ takes place,
\end{itemize}
and
two ``multiple roots'' curves
\begin{itemize}
\item
$S^-$, when $\lambda_1  = \lambda_2 <0$,
\item
$S^+$, when $\lambda_1  = \lambda_2 >0$,
\end{itemize}
which separate the domains with ``node'' and ``focus'' points as well as ``saddle'' and ``saddle-focus'' points.
The curves $S^-$ and $S^+$ are given by the formulas
$$
S^\pm: \;\;\left(\lambda_\pm\right)^3 - A \left(\lambda_\pm\right)^2 -C \lambda_\pm -B =0,
$$
where
$\lambda_\pm$ are the real roots of the equation $3\lambda^2 - 2A\lambda - C=0$.

We note that the domains \textsf{LA} and \textsf{8A} adjoin the domains
\textsf{LQA}:
$\{AC +  B < 0\}\cap\{C< 1 + B^2 +BA\}\cap\{C > A+B+1\}$ and \textsf{8QA}: $\{AC +  B > 0\}\cap\{C > B^2 -1 - BA\}\cap\{C > A+B+1\}$, respectively.
These domains correspond to such  values of $A$ and $C$ that point $O$ has multipliers as follows
\begin{itemize}
\item
\textsf{LQA}: $\lambda_1 <-1, \lambda_2 > |\lambda_3|$ and $\sigma = |\lambda_1\lambda_2| <1$;

\item
\textsf{8QA}: $\lambda_1 <-1, \lambda_2 < |\lambda_3|$ and $\sigma = \lambda_1\lambda_3 <1$.
\end{itemize}
Thus, if map
(\ref{GHM1-8}) would have a homoclinic attractor with point $O$, then this attractor must be a quasiattractor (since $\sigma<1$). Accordingly, we will call these attractors as a {\em discrete Lorenz quasiattractor} if $(A,B)\in\; \mbox{\textsf{LQA}} $ and a {\em discrete
figure-8 quasiattractor} if $(A,B)\in\; \mbox{\textsf{8QA}}$. Note that when $B$ is small these quasiattractors can be labeled also as {\em Henon-like attractors}, see Section~\ref{HenType}.

In the ``spiral-point'' domain, point $O$ has multipliers
$
\lambda_1>1, \lambda_{2,3} = \rho e^{\pm i\varphi}, 0<\rho<1, 0<\varphi<\pi.
$
Therefore, for the corresponding parameters $A$ and $C$, the map
(\ref{GHM1-8}) can have a {\em spiral quasiattractor}, see Figure~\ref{spir1}.
The boundaries of the ``spiral-point'' domain, belonging to the domain $\{C>1-B-A, C<A+B+1\}$, are given by the ``multiple roots'' curves $S^+$ и $S^-$.

In Figure~\ref{fig:8cond1}, the ``Shilnikov-point'' domain is also pointed out. For the corresponding values of $A$ and
$C$, point $O$ has multipliers
$0< \lambda_1<1, \lambda_{2,3} = \rho e^{\pm i\varphi}, \rho>1, 0<\varphi<\pi$.
Thus, for these parameter values,
map (\ref{GHM1-8}) can have the so-called
{\em discrete Shilnikov attractor} \cite{GGS12,GGKT14}, see Fig~\ref{shilattr}.
Also in the domains
$VI$ and $VI^\prime$ this attractor becomes with a saddle fixed point (all multipliers are real).
The topological dimension of such attractors, as well as the dimension of the homoclinic attractors which can exist in the domains I and II, is equal to 2. In our opinion, such attractors are, in fact, quasiattractors, see Introduction.

\begin{figure}[ht]
\centerline{\epsfig{file=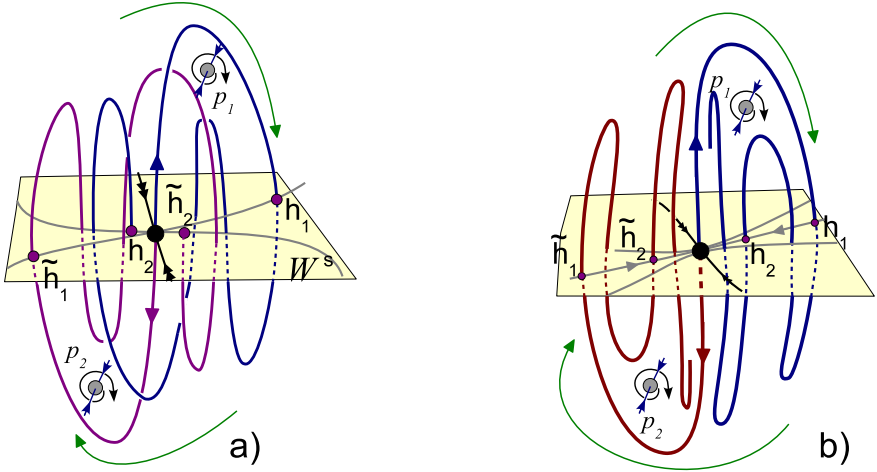, width=14cm
}}
\vspace{-1cm}
\caption{
{\footnotesize Homoclinic configurations in the cases of (а) a spiral attractor; (b) a double figure-8 attractor from the domain \textsf{dLA};
(c) a super-figure-8 attractor from the domain. \textsf{sLA$\&$s8A}.  }
}
\label{spirat}
\end{figure}

Finally, we mention two more interesting domains shown in Figure~\ref{fig:8cond1}, the domains \textsf{D8A} and \textsf{S8A}
corresponding to the values of parameters when map~(\ref{GHM1-8}) can have a pseudohyperbolic homoclinic attractor.

In the domain \textsf{D8A}, the point $O(0,0,0)$ has multipliers $\lambda_1>1, -1 <\lambda_2 <\lambda_3 <0$
and the saddle value $\sigma = |\lambda_1\lambda_2| > 1$ -- the case P3 of multipliers of $O$. Then each of two one-dimensional unstable invariant manifolds (separatrices) of point
$O$ can have a configuration of a figure-8 attractor, see Figure~\ref{spirat}(a).
Therefore, we call this attractor as a {\em discrete double figure-8 attractor}.
The condition $\sigma > 1$ is necessary for this attractor to be pseudohyperbolic and, therefore,
in the domain \textsf{IIIa} such attractor becomes quasiattractor.

In the domain \textsf{S8A}, all three multipliers of $O(0,0,0)$  are positive,
 $\lambda_1>1, 0 <\lambda_3 <\lambda_2 < 1$, and $\sigma = \lambda_1\lambda_2$ is
greater than 1 -- the case P4 of multipliers of $O$.
Then the strong stable  manifold $W^{ss}(O)$ divides the two-dimensional stable
manifold $W^s(O)$ into two components.
If both unstable separatrices would intersect only one of these components of $W^s(O)$, the attractor looks to be similar to the {\em discrete super-Lorenz attractor}, see Introduction. In the case when the unstable separatrices intersect both components of $W^s(O)\backslash W^{ss}(O)$
the corresponding attractor will have a shape of a figure-8 attractor. We will call it a \emph{discrete super-figure-8 attractor}.
With the prefix ``super'' we stress that unlike the discrete Lorenz and figure-8 attractors shown in Figure~\ref{scen2},
the points $p_1$ are $p_2$ are fixed here.
Besides, as always, the condition $\sigma > 1$
is necessary for the pseudohyperbolicity (hence, in the domain \textsf{IIIb}, such attractors are quasiattractors).\footnote{In our numerics,
we were still able to find a super-figure-8 attractor in map (\ref{GHM1-8}) with a cubic nonlinearity (then the map is central-symmetric), see Fig.~\ref{sup8b}. However, despite all our efforts, super-Lorentz attractors
were not found -- probably because of the structure Henon-like maps, since they do not allow  Lorenz-like symmetries.}

\subsection{On a possibility of using the saddle charts for 3D maps. }

We see that the saddle charts can be easily constructed for three-dimensional H\'enon-like maps. However, this method can be adapted for more wide class of three-dimensional maps of form
\begin{equation}
(\bar x, \bar y, \bar z) = D \left(\begin{array}{l} x \\ y \\z \end{array} \right) + F(x,y,z),
\label{M1}
\end{equation}
where $(x,y,z)\in \mathbb{R}^3$, $D$ is a $(3\times 3)$-matrix with $\det D \neq 0$,
$\displaystyle F(0)=0, \frac{\partial D\; F(0)}{\partial (x,y,z)}=0$. Then in general case, by some affine transformation of coordinates, map {\rm (\ref{M1})} can be brought to such a form where the new matrix $D_0$ has a H\'enon form,
\begin{equation}
D_{0} = \left(\begin{array}{lcr} 0, &1,& 0 \\ 0, &0,& 1 \\ B, &C,& A \end{array} \right),
\label{HfM1}
\end{equation}
where $B = \det D $ and $C$ and $A$ are some parameters (linear combinations of the coefficients of matrix $D$).

This is well-known classical result,  and matrix (\ref{HfM1}) is the transposed Frobenius matrix. Such a matrix depends smoothly on parameters  \cite{Arn71} and the corresponding family was called in \cite{Arn71} as a Silvester family of matrices. Note that form (\ref{HfM1}) can not exist only in inclusive trivial cases, for example, when the linear part in (\ref{M1}) splits coordinates (contains lines  like $\bar x = a x$) or so.

\section{Examples of pseudohyperbolic attractors in the generalized H\'enon maps.}  \label{Ex8Lor}

Even a very superficial study of chaotic dynamics of maps of form (\ref{GHM1-8}) shows a wide variety of strange attractors that exist here. This is also true if to say only on homoclinic attractors, i.e. on strange attractors that contains the fixed point $O(0,0,0)$ (and, thus, the corresponding attracting invariant set coincides with the closure of $W^u(O)$).

However, our main problem is to identify among the plurality of these attractors the so-called genuine strange attractors, which, in the  case under consideration, all belong to the class of pseudohyperbolic attractors.
This problem is essentially simplified, if to take into account only those points $O(0,0,0)$ that, themselves, are pseudohyperbolic.
First of all, it is related to the points that are saddles with $\dim W^u(O)=1$ and $\sigma >1$.

%

Thus, from the point of view of the existence of pseudohyperbolic attractors, see Introduction,  only the domains \textsf{LA}, \textsf{8A}, \textsf{d8A}, and \textsf{s8A}, see Fig.~\ref{fig:8cond1}, remain to be most interesting for us. In this section we present some results on the existence of the corresponding pseudohyperbolic strange attractors in maps of form (\ref{GHM1-8}).

Note that there are many different methods for the study of chaotic dynamics in concrete models. For a two parameter analysis, the most suitable and reasonable approaches
are related
to construction of the charts of dynamical regimes and/or the colored Lyapunov diagrams. Indeed, in \cite{GOST05}, discrete
Lorenz attractors were found using the Lyapunov diagrams for the three-dimensional H\'enon map (\ref{3Hen1}).
For example, such attractors were illustrated for $B=0.7; M_1 = 0.0; M_2 = 0.85$, see Fig.1a in \cite{GOST05}, and for $B=0.7; M_1 = 0.0; M_2 = 0.825$, see Fig.1b in \cite{GOST05}.

\begin{figure}[ht]
\centerline{\epsfig{file=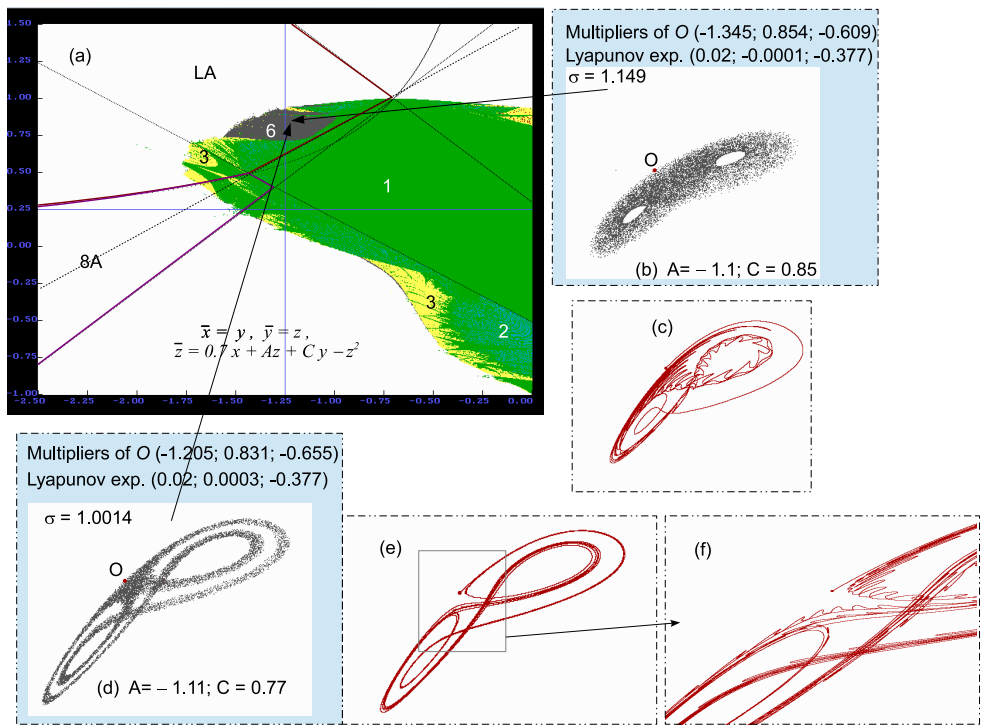, width=14cm
}}
\caption{{\footnotesize Examples of discrete Lorenz attractors in map (\ref{3Hender}) with $B=0.7$: (a) a fragment of the Lyapunov diagram on the (A,C)-parameter plane with a fragment of the saddle chart; (b) and (d) projections of attractors on the $(x,y)$-coordinate plane; (c) and (e) a behavior of one of the unstable separatrices of the fixed point $O$ is shown. Also, an additional information is given: the corresponding values of parameters $A$ and $C$, values of multipliers  and the saddle value of $O$,  and the spectrum of Lyapunov exponents.}}
\label{AL-z2}
\end{figure}

Now we can find such attractors, as they say, ``purposefully''
using the methods developed in the paper. Namely, we consider map (\ref{3Hen1})
in the ``derived'' form
\begin{equation}
 \bar x = y,\; \bar y = z,\; \bar z = B x + Az + Cy - z^2
\label{3Hender}
\end{equation}
and take the ``saddle chart'' for a given fixed $B$, for example, for $B=0.7$,
which is similar to the one as in Figure~\ref{fig:8cond1}.
Further, using the ``saddle chart'', we numerically construct the Lyapunov diagram for orbits starting close to $O(0,0,0)$.
As a result, we obtain the diagram as in Figure~\ref{AL-z2}(a), where, in particular, the dark grey spot 6 intersect the domain \textsf{LA}. This means that for the corresponding values of parameters $A$ and $C$, a discrete Lorenz attractor can be observed.

In particular, projections of such attractors on the $(x,y)$-coordinate plane are shown for $A=-1.1; C= 0.85$ in Fig.~\ref{AL-z2}(b) and for  $A=-1.11; C= 0.77$ in Fig.~\ref{AL-z2}(d), where we point out also values of the multipliers of point $O$, the saddle value of $O$, and the spectra of Lyapunov exponents. In all cases we have that $\sigma>1$ and $\Lambda_1+\Lambda_2>0$. In Figs~\ref{AL-z2}(c) and (e), we show also a behavior of one of the unstable separatrices of the point $O$ (the behavior of another separatrix is symmetric due to the unstable multiplier of $O$ is negative). We see that $W^u(O)$ has in both cases a homoclinic intersection with $W^s(O)$,
although, in the first case, typical zigzags in $W^u$ are seen more clearly than for the second case.

\begin{figure}[ht]
\centerline{\epsfig{file=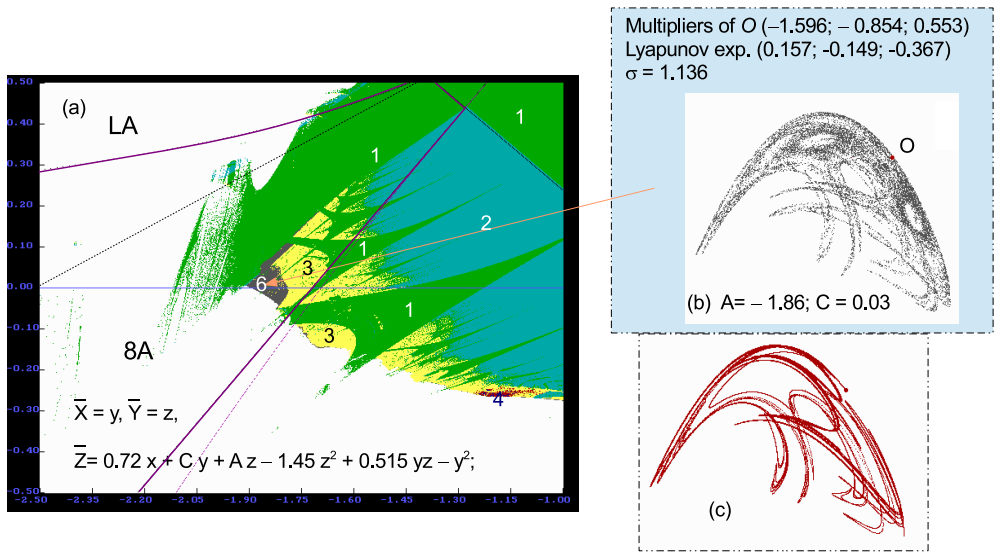, width=14cm
}}
\caption{
{\footnotesize An example of discrete figure-8 attractor in the corresponding generalized H\'enon map: (a) a fragment of the Lyapunov diagram on the (A,C)-parameter plane with a fragment of the saddle chart; (b) the projection of attractor on the $(x,y)$-coordinate plane; (c) a behavior of one of the unstable separatrices of the fixed point $O$ is shown.
%
}
}
\label{8attrnew2}
\end{figure}


The Lyapunov diagram in Figure~\ref{AL-z2} (as well as in the other figures below) is provided in color.
The color codes for the domains are as follows: white -- 
iterations of (all) initial points ``escape to infinity''; green (digit ``1'') -- 
the case of a periodic attractor; light blue (digit ``2'') -- the case of an attractor being a closed invariant curve.
Chaotic attractors can be observed in the yellow (``3''), red (``4''), dark blue (``5'') and dark gray (``6'') domains.
Let $(\Lambda_1, \Lambda_2,\Lambda_3)$ be the spectrum of the Lyapunov exponents of the attractor. Then we assign the colors in the following way: the yellow domains
for the case where $\Lambda_1 >0, \Lambda_2<0,\Lambda_3<0$; the red domains for the case where
$\Lambda_1 >0, \Lambda_2\approx 0,\Lambda_3<0$; the deep blue domains for the case $\Lambda_1 >0, \Lambda_2>0,\Lambda_3<0$.
A special role is played by the dark gray domain -- in this color we denote the case of homoclinic attractors containing the point $O(0,0,0)$.

\begin{figure}[ht]
\centerline{\epsfig{file=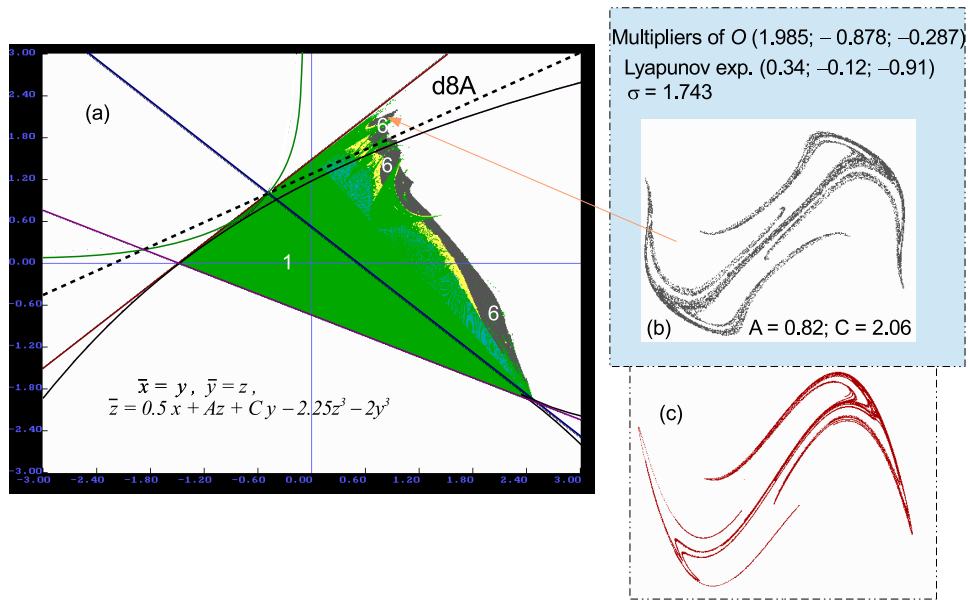, width=14cm
}}
\caption{
{\footnotesize An example of the double figure-8 attractor is shown for the corresponding 3d cubic H\'enon map: (a) a fragment of Lyapunov diagram and saddle chart (we found the attractor in domain 6 of homoclinic chaos lying the \textsf{d8A} region of the saddle chart);
(b) a projection of the attractor on the $(x,y)$-plane; (c) a behavior of one of the unstable separatrices of the fixed point $O$ is shown, other separatrix has the symmetric form due to the cubic map has the central symmetry $x\to-x,y\to-y, z\to-z$.
}
}
\label{d8attr}
\end{figure}

In Figure~\ref{8attrnew2}, the numerical results are provided for the case of the map
$\bar x = y,\; \bar y = z,\; \bar z = 0.72 x + Az + Cy - y^2 + 0.515 yz - 1.45 z^2$. Here one can see that the domain 6 of homoclinic chaos intersects the domain
\textsf{8A}, and, therefore, for certain values of parameters $A$ and $C$  (here $A= -1.86, C= 0.03$)
a discrete figure-8 attractor (Fig.~\ref{8attrnew2}(b)) containing the point $O(0,0,0)$ with multipliers satisfying
condition (\ref{GHM6-8}), is observed. We assume that this attractor is pseudohyperbolic one, since $\Lambda_1 + \Lambda_2 = 0.008 >0$ (but we do not have much confidence in this, since the sum of the first two Lyapunov values is a very small positive value). Note that $W^u(O)$ has a form of a  curve with (infinitely) many zigzags (Fig.~\ref{8attrnew2}(b)) which indicates a possibility of creation of homoclinic tangencies at arbitrarily small perturbations.

\begin{figure}[ht]
\centerline{\epsfig{file=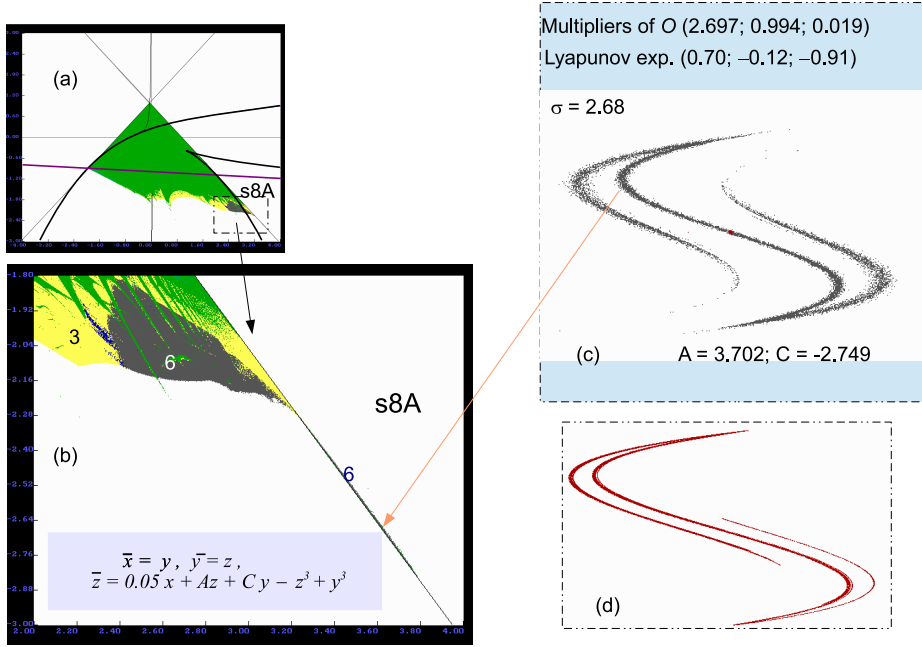, width=14cm
}}
\caption{
{\footnotesize An example of the super figure-8 attractor is shown for the corresponding 3d cubic H\'enon map: (a) and (b) fragments of the Lyapunov diagram and saddle chart;
(c) a projection of the attractor on the $(x,y)$-plane; (c) a behavior of one of the unstable separatrices of the fixed point $O$ is shown.
}
}
\label{sup8b}
\end{figure}

Figure~\ref{d8attr} displays the numerical results for the case of cubic 3d H\'enon map
$\bar x = y,\; \bar y = z,\; \bar z = 0.5 x + Az + Cy - 2 y^3 - 2.25z^3$. Notice that here the domain 6 of homoclinic chaos intersects the domain
\textsf{d8A}, Figure~\ref{d8attr}(a). Hence, for some values of the parameters $A$ and $C$ ($A=0.82; C = 2.06$), a strange attractor
is observed, Figure~\ref{d8attr}(b).
The attractor contains a saddle point with the multipliers $\lambda_1>1, 0> \lambda_2 >\lambda_3>-1$ and the saddle value
$\sigma = |\lambda_1\lambda_2|>1$ -- case P3 of multipliers of $O$. We classify this attractor as {\em discrete double figure-8 attractor}
(compare with Figs.~\ref{8attrnew2} and~\ref{d8attr}).  We consider this attractor as a pseudohyperbolic attractor of new type, it has  $\Lambda_1 + \Lambda_2 \approx 0.22 >0$.  Note that $W^u(O)$ has a form of quite smooth curve at beginning, Fig.~\ref{d8attr}(c), however, certain zigzags appear at its further continuation, Fig.~\ref{d8attr}(d).

One more example of pseudohyperbolic attractor is shown in Fig.~\ref{sup8b} for the case of cubic 3d H\'enon map
$\bar x = y,\; \bar y = z,\; \bar z = 0.05 x + Az + Cy  + y^3 - z^3$. Notice that here the domain 6 of homoclinic chaos is very thin and, nevertheless, it intersects the domain
\textsf{s8A}, see Fig~\ref{sup8b}(a) and its magnification Fig~\ref{sup8b}(b). Then, for some values of parameters $A$ and $C$ ($A = 3.71; C = -2.75$), a strange attractor
is observed, Figuresup8b.
The attractor contains a saddle point with multipliers $\lambda_1>1> \lambda_2 >\lambda_3>0$ and the saddle value
$\sigma = |\lambda_1\lambda_2|>1$ -- case P4 of multipliers of $O$. We classify this attractor as {\em discrete super figure-8 attractor}.
Although this attractor exists for very small (extremely thin) domain of parameters, it demonstrates quite robust properties of  pseudohyperbolicity, e.g. it has $\Lambda_1 + \Lambda_2 = 0.58 >0$.  In Fig.~\ref{sup8b}(d) one of the separatrices of $W^u(O)$ is shown (other separatrix is symmetric due to the central symmetry of the map).

\begin{figure}[ht]
\centerline{\epsfig{file=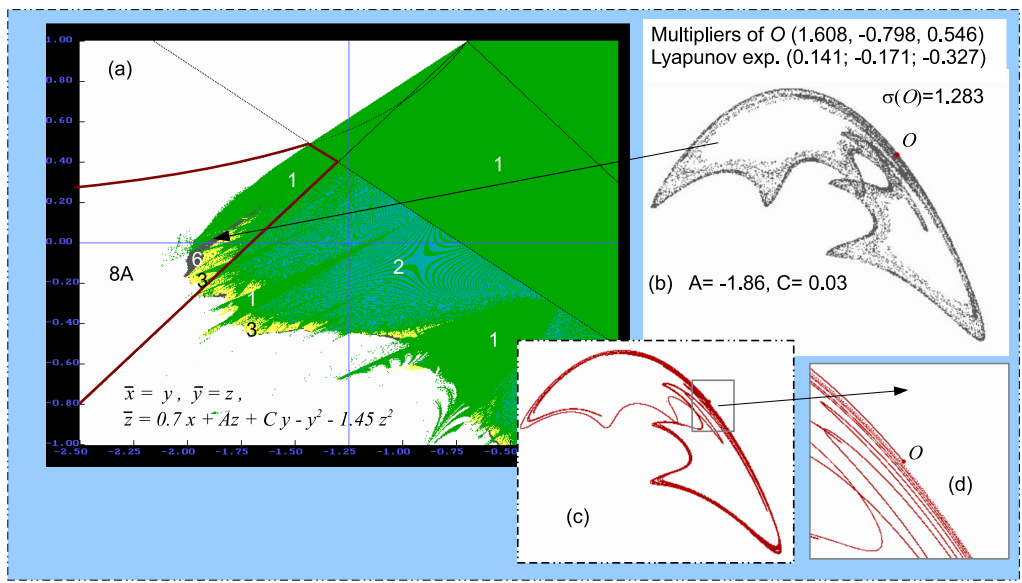, width=14cm
}}
\vspace{-1cm}
\caption{
{\footnotesize An example of 3d H\'enon map  (I) with a discrete figure-8 quasiattractor:
(a) a fragment of the Lyaponov diagram with the saddle chart;
(b) a projection of the attractor on $(x,y)$-plane.
}
}
\label{8quas}
\end{figure}

\section{Examples of quasiattractors in 3d generalized H\'enon maps.}   \label{Exstrattr}

In this section we provide  examples of homoclinic strange attractors which may be interesting, in principle. However, all these attractors are quasiattractors, for one reason or another.
For example, such an attractor is shown in Figure~\ref{8quas} in the case of the 3d H\'enon map  $\bar x = y,\; \bar y = z,\; \bar z = 0.7 x + Az + Cy - y^2 -1.45 z^2$. Here one can see that the domain of homoclinic chaos
intersects the domain
\textsf{8A}, and, therefore, for certain values of the parameters $A$ and $C$  (e.g. $A= -1.86, C= 0.03$),
a discrete figure-8 attractor is observed  containing the point $O(0,0,0)$ with  multipliers satisfying
condition (\ref{GHM6-8}). However, the numerically obtained spectrum of Lyapunov exponents is such that $\Lambda_1 + \Lambda_2 \approx -0.03 < 0$. Thus, the necessary condition (\ref{Lyapcr}) is violated and, hence, the observed discrete figure-8 attractor is, in fact, a quasiattractor.

In Figure~\ref{book-attr}, the results of numerical experiments are shown for the cubic 3d H\'enon map of the form
$\bar x = y,\; \bar y = z,\; \bar z = 0.5 x + Az + Cy - y^3$. Here the chaotic domain intersects
\begin{figure}[ht]
\centerline{\epsfig{file=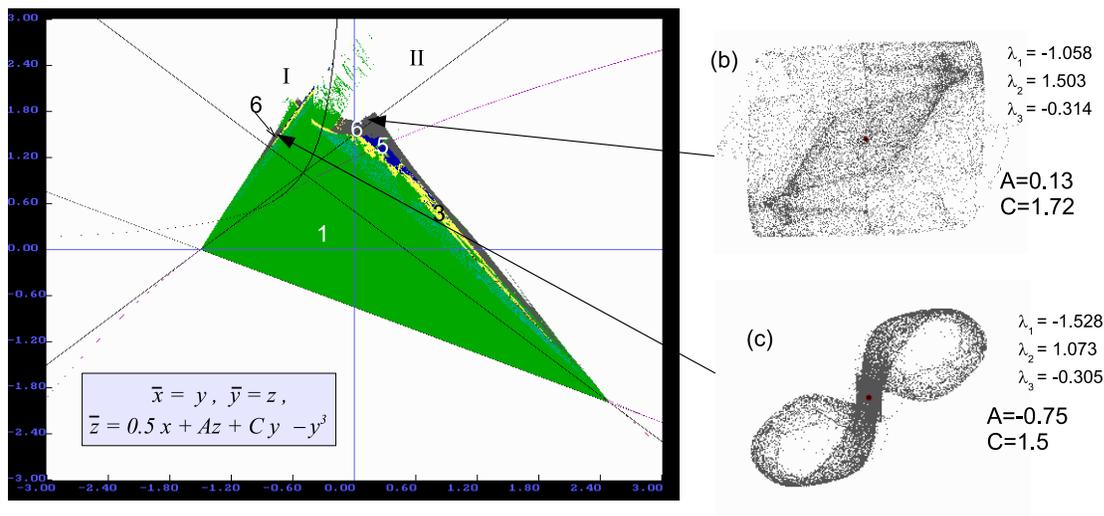, width=16cm
}}
\vspace{-1cm}
\caption{
{\footnotesize (a) The Lyaponov exponents chart with the ``saddle chart'' in the plane of parameters $A$ and $C$ for the corresponding family of cubic 3d
H\'enon maps for $B=0.5$. The values of the multipliers at the saddle and the portraits of homoclinic attractors for
(b) $A= 0.13; C =  1.72$ and (c) $A= -0.75; C =  1.5$.
}
}
\label{book-attr}
\end{figure}
the domains \textsf{I} and \textsf{II}, and,
accordingly, for certain values of
$A$ and $C$, the map possesses homoclinic attractors with saddle fixed points of type
(1,2), i.e. with one-dimensional stable manifold and two-dimensional unstable manifold.
We show two examples of such attractors for (b) $A=0.13, C=1.72$ and
(c) $A=-0.75, C=1.5$.
These attractors have quite different structures
(for example, the first one contains all the three fixed points), but in a neighborhood of
$O(0.0.0)$, their unstable manifolds form infinitely many folds as ``book pages''
(consequently, these attractors could be called as ``book-attractors''; as far as the authors know these attractors has not gotten any name yet).
We attribute these attractors to the class of quasiattractors, however, for other reasons than in the previous case, namely, because of ``multi-folding'' the two-dimensional unstable manifold of $O$, see the Introduction.

\subsection{On discrete spiral attractors in three-dimensional maps.} \label{spirtype}

In this section we provide two examples of discrete homoclinic attractors of spiral type. They are also quasiattractors, since contain the saddle-focus fixed point $O$.
Note that spiral attractors have always special interest beginning from papers by Shilnikov \cite{Sh65,Sh70,Sh86} who discovered the spiral chaos and by Arneodo, Coullet, Tresser \cite{ACT81a,ACT85} who have constructed first examples of systems with strange spiral attractors. We give two different examples of spiral chaos for three-dimensional maps, see Figs.~\ref{spir1} and~\ref{shilattr}.

Fig.~\ref{spir1} illustrates numerical results for the cubic 3d H\'enon map (\ref{GHM1-8}) with
$B=0.5$ and $\tilde f = 0.5 z^3 - 6 zy^2 + 0.5 y^3$. One can see that the dark grey domain 6 of the Lyapunov diagram intersects
the \textsf{``spiral-point''} domain of the saddle chart of point $O$, see Fig.~\ref{spir1}(a). Thus, for certain values of $A$ and $C$ ($A=2.13; C = - 1.29$ in the figure) the map has a spiral attractor
with a saddle-focus (2,1) fixed point $O$, see Fig.~\ref{spir1}(b). A behavior of $W^u(O)$ is illustrated in Fig.~\ref{spir1}(c).

\begin{figure}
\centerline{\epsfig{file=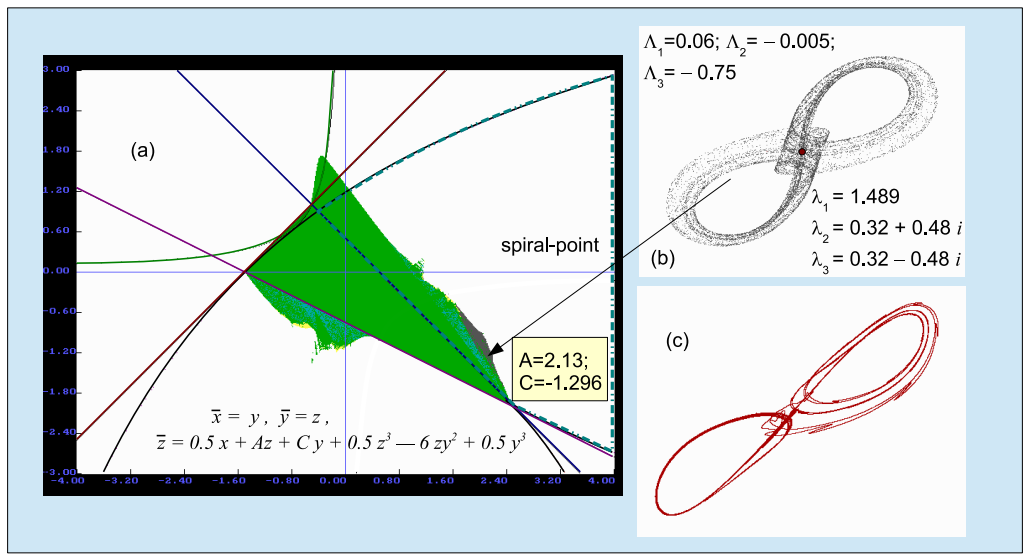, width=14cm
}}
\vspace{-1cm}
\caption{
{\footnotesize (a) The Lyaponov exponents chart with the ``saddle chart'' in the plane of parameters $A$ and $C$ for the corresponding family of cubic 3d
H\'enon maps for $B=0.5$. (b) The values of the multipliers at the saddle and  (c) the portraits of homoclinic attractors for $A=2.13; C = - 1.29$.
}
}
\label{spir1}
\end{figure}

We notice that  two universal and simple scenarios for the onset of chaos
(from a stable fixed point to a strange attractor containing this fixed point) in multidimensional maps were proposed in
\cite{GGS12}, see also \cite{GGKT14}. The common element in both these scenarios is that when a parameter changes the stable fixed point loses the stability
and becomes saddle (but does not belong to the  attractor at the beginning), next, its stable and unstable invariant manifolds intersect, and, eventually,
an attracting set arises containing the saddle fixed point and its unstable manifold. The first scenario of the emergence of the  discrete Lorenz attractor or discrete figure-8 attractor was briefly described in the Introduction, see Fig.~\ref{scen2}.

\begin{figure}[ht]
\centerline{\epsfig{file=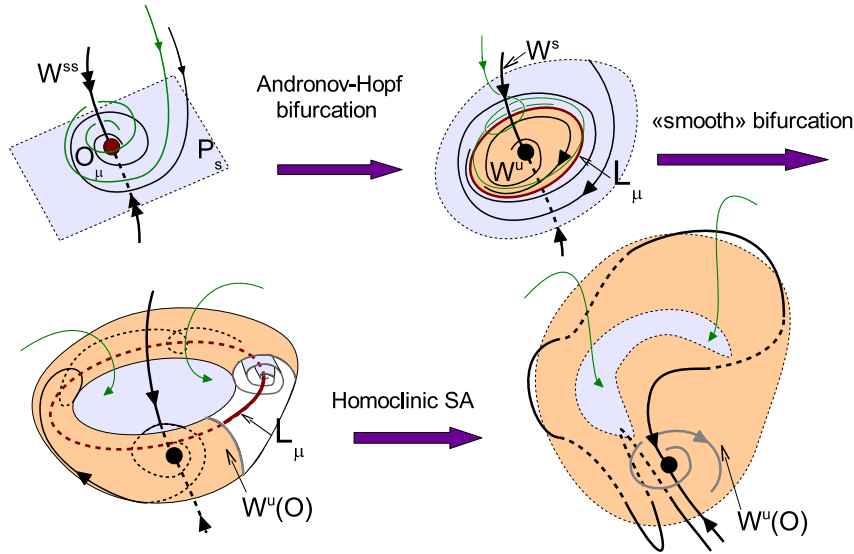, width=14cm
}}
\vspace{-1cm}
\caption{{\footnotesize A sketch of scenario of the emergence of a discrete Shilnikov attractor.}} \label{shildiscr}
\end{figure}

The second scenario
is a scenario of the emergence of a \emph{discrete Shilnikov attractor}.
A sketch of such a scenario for one-parameter families $T_\mu$ of three-dimensional dif\-feo\-mor\-phisms is illustrated in Fig.~\ref{scen2}. 
%
Here the stable fixed point $O_\mu$ loses the stability at $\mu=\mu_1$ via a supercritical (soft) Andronov-Hopf bifurcation:  for
$\mu>\mu_1$ the point becomes a saddle-focus (1,2), i.e. it has the one-dimensional stable and two-dimensional unstable invariant manifolds, and a stable closed invariant curve
$L_\mu$ is born in a neighborhood of the point. Thus,
the curve $L_\mu$ becomes attractor.  Just after the bifurcation, the unstable manifold of $O$ is two-dimensional disk with the boundary $L_\mu$. At further changing $\mu$, the so-called ``smooth'' bifurcation with $L_\mu$ can be happened -- the curve $L_\mu$ becomes of focal type and $W^u(O)$ begins to develop on this curve. Next, the curve
loses the stability, no matter what way. However, it is important that the stable and unstable invariant manifolds of the saddle-focus can intersect and a strange homoclinic attractor can arise containing the saddle-focus
$O_\mu$ and its two-dimensional unstable manifold. We call this attractor as a {\em discrete Shilnikov attractor}.
Notice that a similar scenario of the appearance of a spiral attractor for three-dimensional flows was considered in the paper \cite{Sh86} by L.P.~Shilnikov.

Various realizations of the discrete Shilnikov scenario was considered in \cite{GGKT14}. In the present paper we show only one example of the discrete Shilnikov attractor observed in map (\ref{GHM1-8}) with $\tilde f = - y^2$, see Fig.~\ref{shilattr}. Here the dark grey domain 6 of Lyapunov diagram intersects the \textsf{``Shilnikov-point''} domain of the saddle chart of point $O$,
and, consequently, for some values of $A$ and $C$ ($A=1.43; C = - 1.84$), there is a discrete Shilnikov spiral attractor with a saddle-focus (1,2) fixed point $O$. In this case $W^u(O)$ and $W^s(O)$ intersect that is illustrated by Fig.~\ref{shilattr}(c), where a behavior of one of separatrices of $W^s(O)$ is shown.

\begin{figure}[ht]
\centerline{\epsfig{file=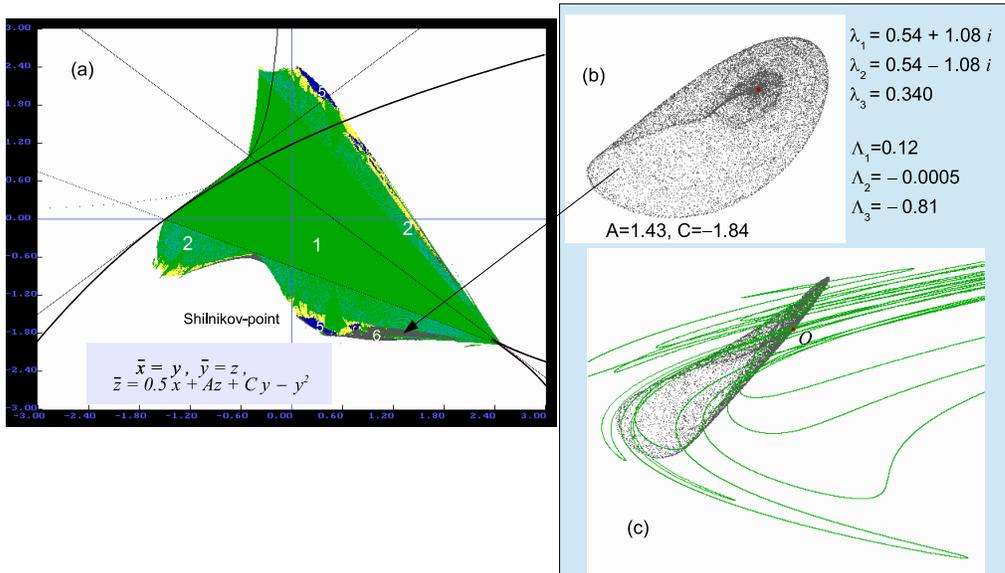, width=14cm
}}
\vspace{-1cm}
\caption{
{\footnotesize (a) The Lyaponov diagram and the saddle chart of $O$ in the plane of parameters $A$ and $C$
for the corresponding family of 3d
H\'enon maps for $B=0.5$.
(b) The projection of attractor onto $(x,y)$-coordinate plane. (c) Homoclinic behavior of one of separatrices of the saddle-focus $O(0,0,0)$ is illustrated (we take here another projection comparing with (b)).
}
}
\label{shilattr}
\end{figure}

In general, the theme of discrete spiral attractors of multidimensional maps seems very interesting and relevant, and we plan to continue appropriate researches in the nearest time.

\subsection{The case $B=0$.} \label{HenType}

We note that the picture of the saddle chart of the point $O$ in the case of map (\ref{GHM1-8})  depends essentially on the value of the Jacobian $B$.
Figure~\ref{fig:8cond1} represents the case $B=0.5$ as well as Figure~\ref{fig:8cond1(0p0)} displays the ``limit'' case $B=0$.
As we can see, there is a huge difference between them. For example, the ``spiral-point'' (and some other) domain disappears for $B=0$.

\begin{figure}[ht]
\centerline{\epsfig{file=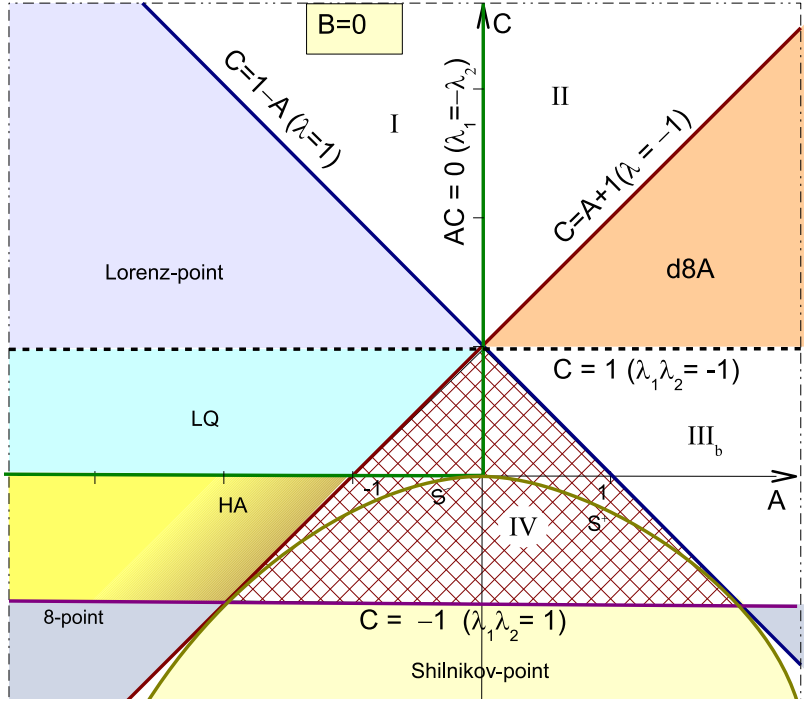, width=12cm
}}
\vspace{-1cm}
\caption{
{\footnotesize The case of two-dimensional maps, $B=0$.  }}
\label{fig:8cond1(0p0)}
\end{figure}

Notice that the case $B=0$ is interesting in itself since the map is reduced to the two-dimensional family of the form
\begin{equation}
 \bar y = z,\; \bar z = Az + Cy + \tilde f(y,z)
\label{2Dmaps}
\end{equation}

\begin{figure}
\centerline{\epsfig{file=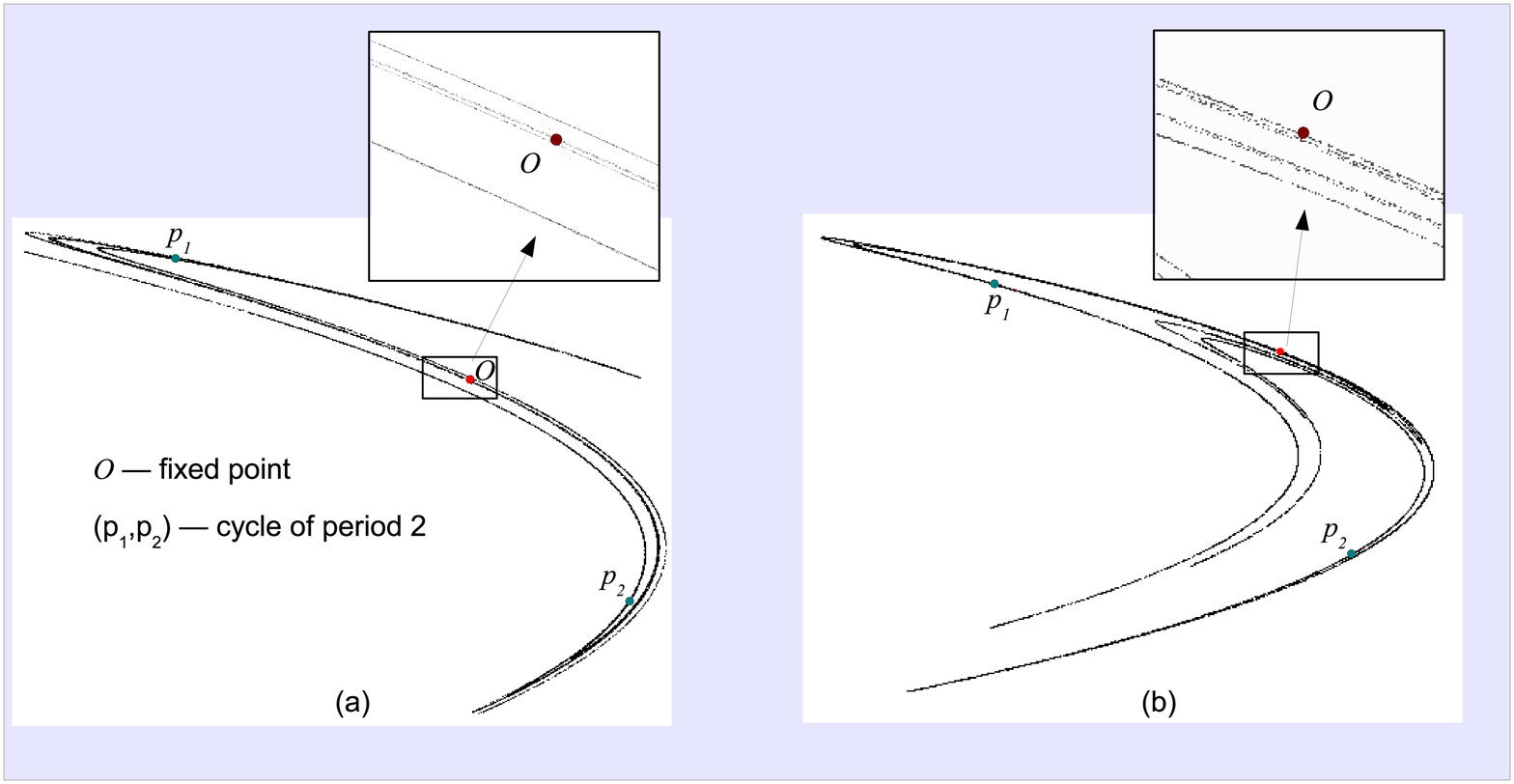, width=14cm
}}
\caption{{\footnotesize Quasiattractors in the two-dimensional H\'enon map (\ref{Henmap2}): (а) H\'enon attractor for $A = -1.92, C = -0.4$;
(b) Lorenz quasiattractor for $A = -1.767, C = 0.3  $.}}
\label{fig:HAandLQ}
\end{figure}

In particular, the H\'enon map,  written in the standard form
\begin{equation}
 \bar y = z,\; \bar z = M + Cy  -z^2,
\label{Henmap}
\end{equation}
can be represented in form (\ref{2Dmaps}) if $D = (C-1)^2 + 4M >0$. In this case map (\ref{Henmap}) has two fixed points with  coordinates $\displaystyle y^\pm =z^\pm = \frac{1}{2}\left(C-1 \pm \sqrt{D}\right)$.
After the shift $y_{new}=y-y^+, z_{new}=z-z^+$ the fixed point
$M^+ = (y^+,z^+)$ is brought into the origin and the map (\ref{Henmap}) takes the form
\begin{equation}
 \bar y = z,\; \bar z = Az + Cy  -z^2,
\label{Henmap2}
\end{equation}
where $A = -2 y^+$. The Jacobian of map (\ref{Henmap2}) is $-C$ and, thus, since the H\'enon map is a diffeomorphism,  attractors can occur only for $|C|<1$. Hence, if $M>0$ (it is the case when map (\ref{Henmap}) can possess an interesting (including chaotic) dynamics) we get $A<0$. Evidently, the map (\ref{Henmap2}) is a member of family (\ref{2Dmaps}), being $\tilde f(y,z) = - z^2$.
From Figure~\ref{fig:8cond1(0p0)} one can see that in the domain $\{A<0;|C|<1\}$ map (\ref{Henmap2}) can have, in our terminology, either a discrete figure-8 quasiattractor
(for $-1<C<0$) or a discrete Lorenz quasiattractor (for $0<C<1$). Examples of such attractors are shown in Figure~\ref{fig:HAandLQ}.

Note that the attractor of Fig.~\ref{fig:HAandLQ}(b) is the same famous H\'enon attractor that was proposed in the paper \cite{H76} by M. H\'enon (the H\'enon's values of parameters correspond to $C=0.3$ and $M=1.4$ for map (\ref{Henmap}) or $C=0.3$ and $A = 0.7 -\sqrt{6.09}\approx -1.767$ for map (\ref{Henmap2})). Thus, M. H\'enon was right when he wrote in \cite{H76} ``We present here a ``reductionist'' approach in which we try to find a model problem which is as simple as possible, yet exhibits the same essential properties as the Lorenz system''.

\section{Conclusion.}  \label{sec:concl}

The paper is devoted to the study of chaotic dynamics of multidimensional systems which include flows with $\dim\geq 4$ and diffeomorphisms with $\dim\geq 3$. We considered three-dimensional diffeomorphisms, and specifically, the three-dimensional generalized H\'enon maps. Note that even in such a setting, the problem of a complete study of strange attractors is quite unrealistic. Therefore, we have restricted this problem to the partial problem of the study of
strange attractors that contain only one fixed point of the map. Moreover, we paid main attention to the so-called pseudohyperbolic strange attractors, which are genuine strange attractors in that sense that each orbit of attractor has a positive maximal Lyapunov exponents and this property is robust, i.e. it holds for all close maps. Then our preliminary considerations suggested that, in principle, three-dimensional maps may have only 5 different types of strange  pseudohyperbolic attractors containing only one fixed point. We call these attractors the discrete Lorenz, discrete figure-8, discrete double-figure-8, discrete super-figure-8, and discrete super-Lorenz attractors. The first four types of attractors were found in the paper. Unfortunately, attractors of the last type,  super-Lorenz ones,
were not found
-- probably because the Henon-like maps do not allow  Lorenz-like symmetries.

We show numerically that our attractors are wild pseudohyperbolic. The proof of ``wildness'' is direct and standard: we construct numerically the one-dimensional stable or unstable manifold of $O$ (for the two-dimensional manifold we take its linear approximation), and look for creation of homoclinic tangencies at varying a parameter. Checking pseudohyperbolic conditions is more delicate. In fact, we prove only that the main condition (\ref{Lyapcr}) holds for the numerically obtained Lyapunov exponents.
However, there exist powerful computer methods, see e.g. \cite{Tucker99,Tucker11,FTV13},
to give the so-called ``computer assisted proof'' of the corresponding results and we hope, with common effort, to get these proofs.
After this we would be able to claim that ``the discrete Lorenz attractor exists'', and also some similarly for ``figure-8'' attractors.

We should notice that our paper continues the study of dynamics of three-dimensional H\'enon-like maps started in a circle of works \cite{GOST05,GGS12,GGOT13,GGKT14}. The possibility of birth of discrete Lorenz attractors at local bifurcations was established in \cite{SST93}; the first example of such attractor
was found in \cite{GOST05}; scenaria of emergence of the discrete Shilnikov and Lorenz attractors in multidimensional maps were described and discussed in \cite{Sh86,SST93,GGS12,GGKT14}; conditions of existence of the discrete Lorenz attractors were derived in \cite{SST93,GOST05,GGOT13} as well as various their examples were found \cite{GGOT13}. We would like to note papers \cite{GGK13,GG15,Kaz14}, where  the discrete Lorenz and figure-8 attractors were discovered in systems of rigid body dynamics. We note also relevant papers \cite{Sat05,TS08} in which important results on pseudohyperbolic properties of periodically perturbed systems with the Lorenz attractor were obtain. In particular, it was shown in \cite{Sat05} that periodic sinks are not created at small such perturbations, and it was proved in \cite{TS08} that the ``discrete super-Lorenz attractor'' (an attractor in the Poincar\'e map) is pseudohyperbolic if the perturbation is sufficiently small.

We hope that our results can  be of interest not only for mathematicians dealing with chaotic dynamics, but also for specialists  from other fields of science.\\

{\bf Acknowledgments.} The authors thank J. Figueras, M. Gonchenko, W. Tucker and D. Turaev for very useful remarks. The paper is partially supported by the grants of RFBR 16-01-00324 and 14-01-00344 and the
RSciF grant No. 14-41-00044. The numerics experiments (in Sections~\ref{Ex8Lor} and~\ref{Exstrattr}) were carried out by the RSciF grant No. 14-12-00811.
AG is supported also by the  basic financial program of Russian Federation Ministry of Education and Science.

\end{document}